\documentclass[11pt]{article}
\usepackage[margin=1in]{geometry} 
\usepackage{amssymb,amsfonts,amsmath,bbm,mathrsfs,stmaryrd}
\usepackage{xcolor}
\usepackage{url}

\usepackage{enumerate}

\usepackage[colorlinks,
             linkcolor=black!75!red,
             citecolor=blue,
             pdftitle={},
             pdfproducer={pdfLaTeX},
             pdfpagemode=None,
             bookmarksopen=true
             bookmarksnumbered=true]{hyperref}

\usepackage{tikz}
\usetikzlibrary{arrows,calc,decorations.pathreplacing,decorations.markings,intersections,shapes.geometric,through,fit,shapes.symbols,positioning,decorations.pathmorphing}

\usepackage{braket}

\usepackage[amsmath,thmmarks,hyperref]{ntheorem}
\usepackage{cleveref}

\creflabelformat{enumi}{#2(#1)#3}

\crefname{section}{Section}{Sections}
\crefformat{section}{#2Section~#1#3} 
\Crefformat{section}{#2Section~#1#3} 

\crefname{subsection}{\S}{\S\S}
\AtBeginDocument{%
  \crefformat{subsection}{#2\S#1#3}%
  \Crefformat{subsection}{#2\S#1#3}%
}

\theoremstyle{plain}

\newtheorem{lemma}{Lemma}[section]
\newtheorem{proposition}[lemma]{Proposition}
\newtheorem{corollary}[lemma]{Corollary}
\newtheorem{theorem}[lemma]{Theorem}

\theoremstyle{nonumberplain}

\theoremstyle{plain}
\theorembodyfont{\upshape}
\theoremsymbol{\ensuremath{\blacklozenge}}

\newtheorem{definition}[lemma]{Definition}
\newtheorem{example}[lemma]{Example}
\newtheorem{remark}[lemma]{Remark}

\crefname{definition}{definition}{definitions}
\crefformat{definition}{#2definition~#1#3} 
\Crefformat{definition}{#2Definition~#1#3} 

\crefname{ex}{example}{examples}
\crefformat{example}{#2example~#1#3} 
\Crefformat{example}{#2Example~#1#3} 

\crefname{remark}{remark}{remarks}
\crefformat{remark}{#2remark~#1#3} 
\Crefformat{remark}{#2Remark~#1#3} 

\crefname{convention}{convention}{conventions}
\crefformat{convention}{#2convention~#1#3} 
\Crefformat{convention}{#2Convention~#1#3}

\crefname{lemma}{lemma}{lemmas}
\crefformat{lemma}{#2lemma~#1#3} 
\Crefformat{lemma}{#2Lemma~#1#3} 

\crefname{proposition}{proposition}{propositions}
\crefformat{proposition}{#2proposition~#1#3} 
\Crefformat{proposition}{#2Proposition~#1#3} 

\crefname{corollary}{corollary}{corollaries}
\crefformat{corollary}{#2corollary~#1#3} 
\Crefformat{corollary}{#2Corollary~#1#3} 

\crefname{theorem}{theorem}{theorems}
\crefformat{theorem}{#2theorem~#1#3} 
\Crefformat{theorem}{#2Theorem~#1#3} 

\crefname{assumption}{assumption}{Assumptions}
\crefformat{assumption}{#2assumption~#1#3} 
\Crefformat{assumption}{#2Assumption~#1#3} 

\crefname{equation}{}{}
\crefformat{equation}{(#2#1#3)} 
\Crefformat{equation}{(#2#1#3)}

\theoremstyle{nonumberplain}
\theoremsymbol{\ensuremath{\blacksquare}}

\newtheorem{proof}{Proof}

\numberwithin{equation}{section}

\newcommand\bC{{\mathbb C}}

\newcommand\bR{{\mathbb R}}
\newcommand\bS{{\mathbb S}}

\newcommand\bZ{{\mathbb Z}}

\DeclareMathOperator{\id}{id}

\newcommand{\qedhere}{\mbox{}\hfill\ensuremath{\blacksquare}}

\title{Residual finiteness for central pushouts}
\author{Alexandru Chirvasitu}

\begin{document}

\date{}

\newcommand{\Addresses}{{
  \bigskip
  \footnotesize

  \textsc{Department of Mathematics, University at Buffalo, Buffalo,
    NY 14260-2900, USA}\par\nopagebreak \textit{E-mail address}:
  \texttt{achirvas@buffalo.edu}

}}

\maketitle

\begin{abstract}
  We prove that pushouts $A*_CB$ of residually finite-dimensional (RFD) $C^*$-algebras over central subalgebras are always residually finite-dimensional provided the fibers $A_p$ and $B_p$, $p\in \mathrm{spec}~C$ are RFD, recovering and generalizing results by Korchagin and Courtney-Shulman. This then allows us to prove that certain central pushouts of amenable groups have RFD group $C^*$-algebras. Along the way, we discuss the problem of when, given a central group embedding $H\le G$, the resulting $C^*$-algebra morphism is a continuous field: this is always the case for amenable $G$ but not in general.
\end{abstract}


\vspace{.5cm}



\section*{Introduction}

Residual finiteness properties have elicited considerable interest, both in the operator algebra literature and in group theory. On the operator-algebraic side one typically considers {\it residually finite-dimensional} (henceforth RFD) $C^*$-algebras, i.e. those whose elements are separated by representations on finite-dimensional Hilbert spaces. The group-theoretic analogue is the concept of a {\it residually finite} (or RF) group, i.e. one whose finite-index normal subgroups intersect trivially (i.e. having ``enough'' finite quotients).

The literature on RFD $C^*$-algebras is rather substantial, as is that on RF groups; so much so, in fact, that it would be impossible to do it justice. For samplings (the best a short note such as this one can do) we direct the reader to, say, \cite{el,MR1301006,bd-popa,adel,MR3250050,krch,ls-rfd,cs,shl-am} (for RFD $C^*$-algebras) and \cite{baums,MR484178,MR311776} or \cite[\S 6.5]{mks}, \cite[Chapters 6, 14, 15]{rob-gp}, \cite[Chapter 2]{csc-cell} (for RF groups), and references therein.

We are concerned here with the types of ``permanence'' properties for residual finiteness as studied, say, in \cite{bd-popa,adel,krch,ls-rfd,cs,shl-am,baums}, to the effect that various types of pushouts (also known as amalgamated free products) of RFD or RF objects are again such. Specifically, the main result of \cite{krch} is that pushouts of separable commutative $C^*$-algebras are RFD. More generally, the main theorem of \cite{cs} proves this for central pushouts $A*_CB$ with $A$ and $B$ separable and {\it strongly} RFD, i.e. such that all of their quotients are RFD.

The present note is motivated in part by the desire to recover these results without the separability and strong RFD-ness assumptions.


The preliminary \Cref{se:prel} recalls some background and sets conventions.

In the short \Cref{se:cntr-alg} we prove \Cref{th:main}, stating that central pushouts of RFD $C^*$-algebras with RFD fibers are RFD.

\Cref{se:dense} is devoted to the problem of when (or whether), given a central group embedding $H\le G$, the resulting embedding $C^*(H)\to C^*(G)$ is a continuous field over $\mathrm{spec}~C^*(H)$ in the sense of \Cref{def:contfld} below. This holds for amenable groups (\Cref{th:amnb-cont}), but not, for instance, for property-(T) groups (\Cref{ex:t-dense}).

\Cref{se:rfgps} centers around the variant of \cite[Theorem 6.9]{shl-am} obtained in \Cref{th:new69}. The latter proves that (the full group $C^*$-algebra of) $G_1*_HG_2$ is RFD provided $G_i$ are amenable residually finite (RF) and $H\le G_i$ is a common central subgroup such that the quotients $G_i/H$ are RF. The former result, on the other hand, assumes that $G_1*_HG_2$ itself is RF.

Although \Cref{th:new69} is formally stronger for that reason, we nevertheless show in \Cref{pr:gicrf} that its hypotheses imply the residual finiteness of the pushout $G_1*_HG_2$. \cite[Theorem 6.9]{shl-am} and \Cref{th:new69} are thus equivalent, albeit non-obviously.

\subsection*{Acknowledgements}

This work was partially supported by NSF grants DMS-1801011 and DMS-2001128. 

I am grateful for much valuable input from Tatiana Shulman, Manuel L. Reyes and Amaury Freslon, as well as the anonymous referee's insightful suggestions. 

\section{Preliminaries}\label{se:prel}

\subsection{Fields of $C^*$-algebras}\label{subse:flds}

All $C^*$-algebras and pushouts are assumed unital. Pushouts over $\bC$ are undecorated, i.e. $A*B$ denotes what in the literature (e.g. \cite{el}) is sometimes referred to as $A*_{\bC}B$. Given a $C^*$ morphism $C\to A$ with $C=C(X)$ commutative, we denote by $A_p$, the {\it fiber} of $A$ at the point $p$ of the spectrum $X$ of $C$:
\begin{equation*}
  A_p:=A/\langle \ker p\rangle,
\end{equation*}
where $p\in X$ is regarded as a character $p:C(X)\to \bC$ and angled brackets denote the ideal generated by the respective set. Similarly, for $a\in A$ we denote by $a_p$ its image through the surjection $A\to A_p$.

We will refer to a central $C^*$ morphism $C\to A$ as a {\it $C$-algebra} $A$ (e.g. \cite[Definition 1.5]{kasp-nov} or \cite[D\'efinition 2.6]{blnch}). Following standard practice (see \cite[D\'efinition 3.1]{blnch} for instance), we have

\begin{definition}\label{def:contfld}
  A $C$-algebra $A$ constitutes a {\it continuous field of $C^*$-algebras} if for every $A$ the norm function
  \begin{equation*}
    \mathrm{spec}~C\ni p\mapsto \|a\|_{A_p}
  \end{equation*}
  is continuous. 
\end{definition}

The function is known to always be {\it upper} semicontinuous, for example by \cite[Proposition 1.2]{rieff-flds}.

\subsection{Group filtrations}\label{subse:fltr}

Following \cite[\S 2.2]{baums}, we need the following notion applicable to a residually finite group $G$.

\begin{definition}\label{def:filt}
  A {\it filtration} on $G$ is a family of finite-index normal subgroups $N_{\alpha}\trianglelefteq G$ with trivial intersection.

  Let $H\le G$ be a subgroup. A filtration $\{N_{\alpha}\}$ of $G$ is {\it an $H$-filtration} provided
  \begin{equation*}
    \bigcap_{\alpha}HN_{\alpha} = H. 
  \end{equation*}
\end{definition}

\section{Pushouts over central subalgebras}\label{se:cntr-alg}

The main result of this section is 

\begin{theorem}\label{th:main}
  Let $A$ and $B$ be two $C^*$-algebras, $C\subseteq A,B$ central $C^*$-embeddings, and assume all corresponding fibers $A_p$ and $B_p$ are RFD for
  \begin{equation*}
    p\in X:=\mathrm{spec}~C.
  \end{equation*}
  Then, $M:=A*_CB$ is RFD.  
\end{theorem}
\begin{proof}
  Because $C$ is central in $A$ and $B$ it is central in $M$. Note the isomorphism
  \begin{equation*}
    M_p\cong A_p*_{C_p}B_p\cong A_p*B_p. 
  \end{equation*}
  For $m\in M$ we have
  \begin{equation}\label{eq:1}
    \|m\| = \sup_{p\in X}\|m_p\|_{M_p}.
  \end{equation}
  by \cite[Proposition 2.8]{blnch} and hence we can approximate the norm of $m$ arbitrarily well with the norms of its images through representations of $M_p\cong A_p*B_p$ as $p$ ranges over $X$. By the RFD-ness assumption on $A_p$ and $B_p$, their coproduct $A_p*B_p$ is again RFD by \cite[Theorem 3.2]{el}. This finishes the proof.
\end{proof}

\begin{remark}\label{re:nonz}
  Note that in fact, in the proof above one does not need the precise norm estimate \Cref{eq:1}: all that is needed is that every $0\ne m\in M$ be non-zero in some quotient $M\to M_p$, $p\in X$.
\end{remark}



\begin{remark}
  The conclusion of \Cref{th:main} cannot hold without the RFD-fiber assumption, as shown in \cite[Proposition 8.3]{shl-am}. That result relies on the construction in \cite{abels} of an amenable RF group $\Gamma$ with a central subgroup $N<\Gamma$ such that $\Gamma/N$ is not RF, and in that context proves that the pushout $C^*(\Gamma)*_{C^*(N)}C^*(\Gamma)$ is not RFD.

  We can see directly, in this case, that the fiber
  \begin{equation*}
    C^*(\Gamma)_p\cong C^*(\Gamma/N)
  \end{equation*}
  at the point $p\in \mathrm{spec}~C^*(N)$ corresponding to the trivial morphism $N\to \{1\}$ is not RFD. Indeed, since $\Gamma/N$ is finitely generated (finitely presented, even) by \cite{abels}, the RFD-ness of $C^*(\Gamma/N)$ would entail the residual finiteness of $\Gamma/N$ (e.g. by \cite[Proposition 2.4]{shl-am}), yielding a contradiction.
\end{remark}

\section{Fibers over dense sets}\label{se:dense}

It will be useful, in working with group $C^*$-algebras, to allow the points $p\in X$ from \Cref{th:main} to range only over a dense subset $Y\subseteq X:=\mathrm{spec}~C$ rather than the entire spectrum. We cannot do this in full generality, as \Cref{eq:1} does not hold for the supremum over only a dense subset $Y\subseteq X$:

\begin{example}\label{ex:cb}
  Let $I=[0,1]$, $M=C_b(I)$ (all bounded functions) and $C=C(I)\subset M$ (continuous functions), both equipped with the supremum norm. The indicator function $m$ of $\{0\}\subset I$ is non-zero, even though its image in every fiber
  \begin{equation*}
    M_p,\ p\in Y:=(0,1]\subset I
  \end{equation*}
  vanishes. For that reason, we do {\it not} have
  \begin{equation*}
    \|m\| = \sup_{p\in Y}\|m_p\|_{M_p}.
  \end{equation*}
\end{example}

What went wrong in the above example is that $M$ was not a continuous field over $C$ in the sense of \Cref{def:contfld}. One might hope that the pathology in \Cref{ex:cb} is at least avoidable for {\it group} $C^*$-algebras; this is not the case, as \Cref{ex:t-dense} below shows. Nevertheless, for amenable groups one can do better (see also \cite[Theorem 3.7]{bl-fr} for a broader discussion of the continuity of pushout fields).

\begin{theorem}\label{th:amnb-cont}
  Let $G_i$, $i=1,2$ be two amenable groups and $H\le G_i$ a common central subgroup. Then,
  \begin{equation*}
    C^*(G_1*_HG_2)\cong C^*(G_1)*_{C^*(H)} C^*(G_2)
  \end{equation*}
  is a continuous field of $C^*$-algebras over $\widehat{H}=\mathrm{spec}~C^*(H)$ in the sense of \Cref{def:contfld}.
\end{theorem}
\begin{proof}
  Since, as noted in \Cref{subse:flds}, for $C:=C^*(H)$ the map
  \begin{equation*}
    C^*(G_1)*_{C}C^*(G_2)=:M \ni m \mapsto \|m_p\|_{M_p}
  \end{equation*}
  is always upper semicontinuous, it remains to prove that for every dense subset
  \begin{equation*}
    Y\subseteq \widehat{H}
  \end{equation*}
  of the spectrum of and every $m\in M$ we have
  \begin{equation*}
    \|m\| = \sup_{p\in Y}\|m_p\|_{M_p}.
  \end{equation*}
  First, fix an arbitrary unitary representation $\rho_2$ of $G_2$ where $H$ acts by scalars, with central character $p_0\in \widehat{H}$. Let
  \begin{equation*}
    Y\ni p_{\lambda}\to p_0
  \end{equation*}
  be a net of characters from $Y$ converging to $p_0$ and set
  \begin{equation*}
    q_{\lambda}:=p_{\lambda}p_0^{-1}\to 1\in \widehat{H}
  \end{equation*}
  be the corresponding ``error''.

  Because
  \begin{equation*}
    G_i/H,\ i=1,2 
  \end{equation*}
  are amenable homogeneous spaces of $G_i$ respectively in the sense of \cite[\S 2]{bekka-amn}, the trivial representation ${\bf 1}_{G_2}$ is weakly contained in the induced representation $\mathrm{Ind}_H^{G_2}{\bf 1}_H$. Since $q_{\lambda}\to {\bf 1}_{H}$ and induction is continuous \cite[Theorem F.3.5]{bhv}, we have
  \begin{equation*}
    \mathrm{Ind}_H^{G_2}q_{\lambda}\to \mathrm{Ind}_H^{G_2}{\bf 1}_H\to {\bf 1}_{G_2}
  \end{equation*}
  in the Fell topology, and hence 
  \begin{equation}\label{eq:7}
    \rho_2\otimes \mathrm{Ind}_H^{G_2}q_{\lambda}
    \to
    \rho_2\otimes{\bf 1}_{G_2} \cong \rho_2.
  \end{equation}
  Because $H$ acts by $p_0$ in $\rho_2$, it acts via $p_\lambda=q_{\lambda}p_0$ in the left hand side of \Cref{eq:7}. In short, we can Fell-approximate $\rho_2$ with unitary representations where $H$ acts by characters from $Y$.

  Now fix $m\in M$. According to \cite[Proposition 2.8]{blnch}, the norm of $m$ is achieved in some unitary representation $\rho$ of $G:=G_1*_HG_2$ where $H$ acts by scalars, with central character $p_0\in \widehat{H}$. Working only with non-degenerate (indeed, unital) representations isomorphic to their own $\aleph_0$-multiples, \cite[Lemma 2.4]{fell-wk1} implies that it will suffice to approximate $\rho$ arbitrarily well in the Fell topology \cite[Appendix F.2]{bhv} with unitary representations where $H$ operates with central characters belonging to the dense subset $Y\subseteq \widehat{H}$.

  In turn, in order to achieve the above it is enough to approximate the restrictions $\rho|_{G_i}$ with representations where $H$ acts via elements of $Y$. This, however, is what the first part of the proof does.
\end{proof}

Over the course of the proof of \Cref{th:amnb-cont} we have obtained

\begin{proposition}\label{pr:amnbl-dense}
  Let $H\le G$ be a central subgroup of an amenable group. Then, for every dense subset
  \begin{equation*}
    Y\subseteq \widehat{H}=\mathrm{spec}~C^*(H)
  \end{equation*}
  the canonical map
  \begin{equation}\label{eq:8}
    C^*(G)\to \prod_{p\in Y}C^*(G)_p
  \end{equation}
  is one-to-one.
  \qedhere
\end{proposition}

In other words, a variant of \cite[Proposition 2.8]{blnch} with only a {\it dense} subset of the spectrum rather the entirety of it.

To prepare the ground for \Cref{ex:t-dense}, note that by \cite[Theorem F.4.4]{bhv} \Cref{eq:8} is an embedding if and only if the $G$-representations where $H$ acts by characters in $Y$ form a dense set in the Fell topology. This cannot possibly happen if

\begin{itemize}
\item $G$ has the Kazhdan property (T) (e.g. \cite[\S 1.1]{bhv}), and hence the trivial representation is isolated in the unitary dual;
\item the dense subset $Y\subset \widehat{H}$ does not contain the trivial element. 
\end{itemize}

To construct such examples all we need is a property-(T) group $G$ with an {\it infinite} central subgroup $H$, whereupon we can simply take
\begin{equation*}
  Y=\widehat{H}\setminus\{1\}. 
\end{equation*}

\begin{example}\label{ex:t-dense}
  Let $G$ be the universal cover $\widetilde{Sp_{4}(\bR)}$ of the $4\times 4$ real symplectic group. It is shown in \cite[Theorem 6.8]{MR3552017} that $G$ has property (T) (indeed, even the stronger property (T$^*$); note that the authors of that paper denote $Sp_4$ by $Sp_2$).
  
  The fundamental group of $Sp_4(\bR)$ is $\bZ$, so we can simply take that copy of $\bZ$ as the infinite central subgroup $H<G$.

  Finally, if a {\it discrete} group is desired then one can simply take the preimage through
  \begin{equation*}
    \widetilde{Sp_4(\bR)}\to Sp_4(\bR)
  \end{equation*}
  of any lattice in the latter (it will again have property (T) by \cite[Theorem 1.7.1]{bhv}). See also \cite[\S 1.7]{bhv} for a discussion of property (T) permanence under passage to universal covering groups.
\end{example}

\section{Central pushouts of RF groups}\label{se:rfgps}

The present section attempts to prove a slightly more general version of \cite[Theorem 6.9]{shl-am}. The difference is that we only assume that $G_i$ are individually RF rather than assuming that $G_1*_HG_2$ is.

\begin{theorem}\label{th:new69}
  Let $G_i$, $i=1,2$ be two amenable groups and $H\le G_i$ a common central subgroup such that
  \begin{itemize}
  \item each $G_i$ is RF;
  \item each $G_i/H$ is RF. 
  \end{itemize}
  Then, $G_1*_HG_2$ is RFD. 
\end{theorem}
\begin{proof}
  We will obtain the result as an application of \Cref{th:amnb-cont}, with
  \begin{equation*}
    A=C^*(G_1),\ B=C^*(G_2)\text{ and } C=C^*(H). 
  \end{equation*}
  According to that result, what we have to argue is that the fibers $A_p$ and $B_p$ are RFD for a dense set of points $p$ in the spectrum
  \begin{equation*}
    \widehat{H}=\text{Pontryagin dual of }H = \mathrm{spec}\; C^*(H). 
  \end{equation*}
  This is precisely what \Cref{cor:is-rfd} does, finishing the proof.
\end{proof}


\begin{lemma}\label{le:sandwich}
  Let $H\le G$ be a central inclusion such that $G$ and $G/H$ are both RF. Then, the characters $p:H\to \bS^1$ whose kernels
  \begin{equation*}
    N=\ker p
  \end{equation*}
  give rise to RF quotients $G/N$ form a dense subset of the Pontryagin dual $\widehat{H}$.
\end{lemma}
\begin{proof}
  Since $G$ is RF the normal finite-index subgroups $G_\alpha\trianglelefteq G$ have trivial intersection. Additionally, since $G/H$ is RF, $G$ has an $H$ filtration in the sense of \Cref{def:filt} (e.g. \cite[Proposition 3.3]{shl-am}) and hence
  \begin{equation}\label{eq:5}
    \bigcap_{\alpha}HG_{\alpha} = H.
  \end{equation}

  Since
  \begin{equation*}
    H_{\alpha}:=G_{\alpha}\cap H\le H
  \end{equation*}
  have trivial intersection, the union of the duals
  \begin{equation*}
    \widehat{H/H_{\alpha}}\subseteq \widehat{H}
  \end{equation*}
  is a dense subgroup. We claim that any $p\in \widehat{H/H_{\alpha}}$ will meet the requirements in the statement; proving this will achieve the desired conclusion, so it is the task we turn to next.

  Fix such a character $p:H\to \bS^1$, factoring through some
  \begin{equation*}
    H\to H/H_{\alpha_0},
  \end{equation*}
  and let $N\le H$ be its kernel. Then, by its very construction, $N$ will contain $H_{\alpha_0}=H\cap G_{\alpha_0}$. This means that
  \begin{equation*}
    N G_{\alpha_0}\cap H\subseteq N
  \end{equation*}
  (which is then an equality), and hence the intersection
  \begin{equation*}
    \bigcap_{\alpha}N G_{\alpha}\subseteq H
  \end{equation*}
  (where the latter inclusion uses \Cref{eq:5}) cannot possibly contain $N$ {\it strictly}. In conclusion we have
  \begin{equation*}
    \bigcap_{\alpha}NG_{\alpha} = N,
  \end{equation*}
  meaning that the filtration $\{G_{\alpha}\}$ is compatible with $N$ (i.e. an $N$-filtration) and hence $G/N$ is RF.
\end{proof}

\begin{remark}\label{re:kg}
  The RF requirements in \Cref{th:new69} are both crucial:

  On the one hand, \cite[\S 8]{shl-am} recalls the example given by Abels in \cite{abels} of a central inclusion $\bZ<G$ into an RF amenable group such that $G/\bZ$ is not RF.
  
  On the other hand, \cite{cmpb} contains an example (attributed there to C. Kanta Gupta) of a non-residually finite amenable group $G$ whose quotient by an order-two normal subgroup $K\trianglelefteq G$ is residually finite. Centrality is easy to arrange, since the centralizer of $K$ in $G$ will have finite index in the latter.
\end{remark}

Although \Cref{le:sandwich} ensures that the quotients $G/N$
for
\begin{equation*}
  N=\ker(p:H\to \bS^1)
\end{equation*}
are RF, we have yet to prove that the resulting fiber algebras $C^*(G)/\langle h-p(h),\ h\in H\rangle$ are RFD. 

\begin{lemma}\label{le:is-rfd}
  Let $H\le G$ be a central subgroup, $p\in \widehat{H}$ a character of finite order, and $N=\ker p$ its kernel in $H$. If $G/N$ is RFD then so is the fiber $C^*$-algebra
  \begin{equation}\label{eq:6}
    C^*(G)_p:=C^*(G)/\langle h-p(h),\ h\in H\rangle
  \end{equation}
  corresponding to $p$ is RFD.  
\end{lemma}
\begin{proof}
  Note that $C^*(G)_p$ is precisely the fiber of $C^*(G/N)$ at the character $\overline{p}$ induced by $p$ on the quotient (finite cyclic) group $H/N$: indeed, denoting images of elements $g\in G$ in $G/N$ by $\overline{g}$, we have
  \begin{equation*}
    C^*(G)_p = C^*(G)/\langle h-p(h),\ h\in H\rangle = C^*(G/N)/\left\langle \overline{h}-\overline{p}\left(\overline{h}\right),\ \overline{h}\in H/N\right\rangle = C^*(G/N)_{\overline{p}}
  \end{equation*}
because the kernel of $C^*(G)\to C^*(G/N)$ is generated by $h-1$, $h\in N$, which are already contained in the ideal we are modding out in \Cref{eq:6}. 

  For that reason, we may as well assume that
  \begin{itemize}
  \item $N$ is trivial, and hence
  \item $H$ is finite cyclic;
  \item $G$ is RFD. 
  \end{itemize}
  But now note that $C^*(G)_p$ is a fiber of the RFD $C^*$-algebra over the {\it finite-dimensional} central subalgebra $C^*(H)\le C^*(G)$. In general, a $C^*$-algebra $A$ will break up as a product of the fibers $A_p$ over a finite-dimensional central $C^*$-subalgebra $C\le A$, by simply cutting $A$ with the minimal projections in $C$.

  In particular, under our assumptions the fiber $C^*(G)_p$ is a Cartesian factor of $C^*(G)$, and hence the RFD-ness of the latter entails that of the former.
\end{proof}

\begin{corollary}\label{cor:is-rfd}
Under the hypotheses of \Cref{le:sandwich} the characters $p:H\to \bS^1$ for which the fiber \Cref{eq:6} is RFD form a dense subset of $\widehat{H}$.   
\end{corollary}
\begin{proof}
  Indeed, the characters $p$ in the proof of \Cref{le:sandwich} are of finite order, and hence \Cref{le:is-rfd} applies.
\end{proof}

\subsection{Recovering residual finiteness for a pushout}

Recall that \cite[Theorem 6.9]{shl-am} assumes $G_1*_CG_2$ is RF, whereas \Cref{th:new69} only requires that $G_i$, $i=1,2$ be RF individually (along with $G_i/H$). We argue here that that distinction is only apparent:

\begin{theorem}\label{pr:gicrf}
Let $H\le G_i$, $i=1,2$ be a common central subgroup such that $G_i$ and $G_i/H$ are all RF. Then, $G_1*_H G_2$ is RF.   
\end{theorem}
\begin{proof}
  Note first that the case $G_1=G_2$ is clear: indeed, the assumptions that $G$ and $G/H$ are RF then show that $G$ has an $H$-filtration. The two copies of that filtration in the two copies of $G$ are then $(H,H,\id)$-compatible in the sense of \cite[\S 2.2]{baums}, and hence $G*_H G$ is RF by \cite[Proposition 2]{baums}.

  It thus remains to reduce the problem to the case $G_1=G_2$. To do this, consider the {\it tensor product}
  \begin{equation*}
    G:=G_1\otimes_HG_2
  \end{equation*}
  defined by identifying the two copies of $H$ in $G_1\times G_2$; in other words, $G_1\otimes_HG_2$ is $G_1*_HG_2$ modulo the relations making the elements of
  \begin{equation*}
    G_1\text{ and }G_2\; \le\; G_1*_HG_2
  \end{equation*}
  commute.

  We then have
  \begin{equation*}
    G/H\cong (G_1/H)\times (G_2/H),
  \end{equation*}
  which is thus RF by assumption. On the other hand, if we show that $G*_H G$ itself is RF then so is
  \begin{equation*}
    G_1*_HG_2\le G*_HG
  \end{equation*}
  (where the inclusions $G_i\le G=G_1\otimes_HG_2$ are the obvious ones).

  To summarize, we have thus far
  \begin{itemize}
  \item observed that the conclusion holds when $G_1=G_2$ (equal to a common group $G$, say); 
  \item reduced the problem to its instance for $G=G_1\otimes_HG_2$, modulo
  \item the desired hypothesis that that $G$ is RF. 
  \end{itemize}

  In conclusion, all that remains to be proven is that under our hypotheses $G:=G_1\otimes_HG_2$ is indeed RF; we relegate this to \Cref{le:otimes}.
\end{proof}

\begin{lemma}\label{le:otimes}
  Let $H\le G_i$, $i=1,2$ be a common central subgroup such that $G_i$ and $G_i/H$ are all RF. Then, $g:=G_1\otimes_H G_2$ is RF.
\end{lemma}
\begin{proof}
  If $\{G_{i,\alpha}\}$ are $H$-filtrations of $G_i$ respectively for $i=1,2$ then the images $G_{\alpha}$ of
  \begin{equation*}
    G_{1,\alpha}\times G_{2,\alpha}\le G_1\times G_2
  \end{equation*}
  through the surjection
  \begin{equation*}
    G_1\times G_2\to G=G_1\otimes_HG_2
  \end{equation*}
  identifying the two copies of $H$ will form an $H$-filtration for $G$.
\end{proof}


\def\polhk#1{\setbox0=\hbox{#1}{\ooalign{\hidewidth
  \lower1.5ex\hbox{`}\hidewidth\crcr\unhbox0}}}

\addcontentsline{toc}{section}{References}

\Addresses

\end{document}